\documentclass{birkau}

\usepackage{amssymb}
\usepackage{amsmath}
\usepackage{enumerate}
\usepackage{amsfonts}
\usepackage{color}
\usepackage{graphicx}
\usepackage{multicol}
\usepackage{multirow}
\usepackage{stmaryrd}
\usepackage{mathtools}
\usepackage{bussproofs}		

\usepackage{tikz}
\tikzstyle{bull}=[circle,draw=black,fill=black!80]
\usetikzlibrary{matrix}

\usepackage[hidelinks]{hyperref}
\hypersetup{colorlinks=true,citecolor=blue,linkcolor=blue}

\allowdisplaybreaks  					

\numberwithin{equation}{section}

\theoremstyle{plain}
\newtheorem{theorem}{Theorem}[section]
\newtheorem{lemma}[theorem]{Lemma}
\newtheorem{proposition}[theorem]{Proposition}
\newtheorem{corollary}[theorem]{Corollary}

\theoremstyle{definition}
\newtheorem{definition}[theorem]{Definition}
\newtheorem{example}[theorem]{Example}
\newtheorem{remark}[theorem]{Remark}

\newcommand{\nc}{{\,\mid\!\sim\,\!\!}}

\newcommand{\Fi}{\mathsf{Fi}}
\newcommand{\Fig}{\mathrm{Fi}}

\newcommand{\Id}{\mathsf{Id}}

\newcommand{\ph}{\varphi}
\newcommand{\p}{\mathcal{P}}
\newcommand{\G}{\mathcal{G}}
\newcommand{\C}{\mathrm{C}}
\newcommand{\Cs}{\mathcal{C}}
\newcommand{\Ll}{\mathcal{L}}
\newcommand{\Var}{Var}
\newcommand{\Fm}{Fm}
\newcommand{\Hom}{\mathrm{Hom}}
\newcommand{\Seq}{\mathrm{Seq}}

\renewcommand{\S}{\mathcal{S}}
\newcommand{\V}{\mathbb{V}}

\newcommand{\vS}{\vdash_\S}
\newcommand{\Th}{\mathsf{Th}}
\newcommand{\vtl}{\rhd}
\newcommand{\GMod}{\mathrm{\textbf{GMod}}}

\newcommand{\Alg}{\mathrm{Alg}}
\newcommand{\K}{\mathrm{K}}

\newcommand{\Con}{\mathrm{Con}}

\newcommand{\DN}{\mathbb{DN}}
\newcommand{\dn}{\mathrm{DN}}

\newcommand{\FiF}{\mathsf{Fi}_{\mathsf{F}}}

\newcommand{\Sl}{\S_\K^{\leq}}

\begin{document}


\title[Selfextensional logics with a $\dn$-term]{Selfextensional logics with a distributive nearlattice term}

\author[L. J. Gonz\'alez]{Luciano J. Gonz\'alez}
\address{Fac. Cs. Exactas y Naturales\\
Universidad Nacional de La Pampa\\Santa Rosa, 6300\\Argentina}
\email{lucianogonzalez@exactas.unlpam.edu.ar}

\thanks{This work was partially supported by Universidad Nacional de La Pampa (Fac. de Cs. Exactas y Naturales) under the grant P.I. 64 M, Res. 432/14 CD; and also by Consejo Nacional de Investigaciones Cient\'ificas y T\'ecnicas (Argentina) under the grant PIP 112-20150-100412CO}

\subjclass{03G27, 03B22,	03G25, 06A12}

\keywords{Selfextensional logics, distributive nearlattices, logics based on partial orders, abstract algebraic logic}

\begin{abstract}
We define when a ternary term $m$ of an algebraic language $\Ll$ is called a \textit{distributive nearlattice term} ($\dn$-term) of a sentential logic $\S$. Distributive nearlattices are ternary algebras generalising Tarski algebras and distributive lattices. We characterise the selfextensional logics with a $\dn$-term through the interpretation of the DN-term in the algebras of the algebraic counterpart of the logics. We prove that the canonical class of algebras (under the point of view of Abstract Algebraic Logic) associated with a selfextensional logic with a $\dn$-term is a variety, and we obtain that the logic is in fact fully selfextensional.
\end{abstract}

\maketitle


\section{Introduction}

This paper is motivated by the results and ideas given in \cite{Ja05} and \cite{Ja06}. In \cite{Ja05}, selfextensional finitary logics with a binary term $\to$ satisfying the deduction-detachment theorem are studied. There, these logics are characterised as logics $\S$ for which there is a class of algebras $\K$ such that the equations defining Hilbert algebras (also called positive implication algebras) hold for $\to$ and the following condition is satisfied:
\begin{equation*}
\begin{split}
\ph_0,\dots,\ph_{n-1}\vdash_\S\ph \iff &(\forall A\in\K)(\forall h\in\Hom(\Fm,A))\\
&h(\ph_0\to(\dots\to(\ph_{n-1}\to\ph)\dots))=1.
\end{split}
\end{equation*}
Similar results are obtained in \cite{Ja06}. There, selfextensional finitary logics with a conjunction term $\wedge$  are characterised as logics $\S$ for which there is a class of algebras $\K$ such that the semilattice equations are satisfied for $\wedge$ and the following condition holds:
\begin{equation*}
\begin{split}
\ph_0,\dots,\ph_{n-1}\vdash_\S\ph \iff &(\forall A\in\K)(\forall h\in\Hom(\Fm,A))\\
&h(\ph_0)\wedge\dots\wedge h(\ph_{n-1})\leq h(\ph).
\end{split}
\end{equation*}

As we can notice from the definitions above, the two kinds of sentential logic are characterised through the behaviour of the interpretation of the implication $\to$ and the conjunction $\wedge$ in the corresponding classes of algebras $\K$.

The notion of \textit{distributive nearlattice} can be presented in two different and equivalent ways. They can be defined as join-semilattices with some extra properties and can be defined as algebras with only one ternary operation satisfying some identities. The two different ways to consider distributive nearlattices are useful for various purposes. After the formal definition of distributive nearlattice (Definition \ref{def: DN-alg}), we will see (Remark \ref{rem:generalisation of DN}) that the variety of distributive nearlattices is a natural generalisation of both the variety of Tarski algebras (also called implication algebras) and the variety of distributive lattices.

The primary aim of this paper is to propose a definition of when a ternary term $m$ of an algebraic language $\Ll$ can be considered a \textit{distributive nearlattice term} (\textit{DN-term} for short) for a sentential logic $\S$. We present some syntactical properties (Section \ref{sec:DN-based logic})  on a sentential logic $\S$ concerning a ternary term $m$ such that, when $m$ is interpreted in every algebra $A$ of the algebraic counterpart of the logic $\S$, the $\{m\}$-reduct $\langle A,m^A\rangle$ will be a distributive nearlattice.

We show that selfextensional logics with a distributive nearlattice term $m$ can be characterised as logics $\S$ for which there exists a class of algebras $\K$ such that the $\{m\}$-reducts of the algebras of $\K$ are distributive nearlattices and the consequence relation of $\S$ can be defined using the partial order induced by the term $m$ on the algebras of $\K$ (Section \ref{sec:DN-based logic}). 

In Section \ref{sec:5}, given a selfextensional logic $\S$ with a $\dn$-term (and with theorems), we consider two sentential logics associated with the canonical class of algebras of $\S$; namely, \textit{the logic preserving degrees of truth} and \textit{the truth-preserving logic}. We show some properties of these logics, and we present some sufficient conditions for these logics to coincide with the original logic $\S$.

\section{Preliminaries}\label{sec 2}

In this section, we introduce some basic concepts and results of Abstract Algebraic Logic (AAL) needed for what follows in the paper and we present the algebraic theory of nearlattices. Our main references for AAL are \cite{Cz01,FoJa09,FoJaPi03,Fo16} and for the theory of nearlattice are \cite{ChaHaKu07,Hi80,ChaKo08}.

\subsection{Abstract Algebraic Logic}\label{subsec:AAL}

Let $\Ll$ be an algebraic language (or algebraic similarity type). We denote by $\Fm(\Ll)$ the absolutely free algebra of type $\Ll$ with a denumerable set $\Var$ of propositional variables as the set of generators. The algebra $\Fm(\Ll)$ is called the \textit{algebra of formulas} of type $\Ll$ and its elements are called \textit{formulas}. When there is no danger of confusion, we write $\Fm$ instead of $\Fm(\Ll)$.

A \textit{sentential logic} (also called \textit{deductive system} in AAL) of type $\Ll$ is a pair $\S=\langle\Fm,\vdash_{\S}\rangle$ where $\Fm$ is the algebra of formulas of type $\Ll$ and $\vdash_{\S}{\subseteq}\,\p(\Fm)\times\Fm$ is a relation satisfying the following properties: for all $\Gamma,\Delta\subseteq\Fm$ and $\ph\in\Fm$ (as usual we write $\Gamma\vdash_\S\ph$ for $(\Gamma,\ph){\in}\vdash_{\S})$,
\begin{enumerate}[(S1)]
	\item if $\ph\in\Gamma$, then $\Gamma\vdash_\S\ph$;
	\item if $\Gamma\vdash_\S\ph$ and $\Gamma\subseteq\Delta$, then $\Delta\vdash_\S\ph$;
	\item if $\Gamma\vdash_\S\ph$ and for every $\gamma\in\Gamma$, $\Delta\vdash_\S\gamma$, then $\Delta\vdash_\S\ph$;
	\item if $\Gamma\vdash_\S\ph$, then there is a finite $\Gamma_0\subseteq\Gamma$ such that $\Gamma_0\vS\ph$;
	\item if $\Gamma\vdash_\S\ph$, then $\sigma[\Gamma]\vS\sigma(\ph)$ for all substitution $\sigma\in \Hom(\Fm,\Fm)$.
\end{enumerate}
The relation $\vS$ is called the \textit{consequence relation} of $\S$. A set $\Gamma\subseteq\Fm$ is called a \textit{theory} of $\S$ ($\S$-\textit{theory}, for short) if is closed under the consequence relation of $\S$, that is, for every formula $\ph\in\Fm$, if $\Gamma\vS\ph$, then $\ph\in\Gamma$. Let us denote by $\Th(\S)$ the collection of all $\S$-theories. It is easy to see that $\Th(\S)$ is an algebraic closure system on $\Fm$ and the closure operator associated with $\Th(\S)$, which is denoted by $\C_\S$, is defined as:
\[
\ph\in\C_\S(\Gamma)\iff \Gamma\vS\ph
\]
for all $\Gamma\cup\{\ph\}\subseteq\Fm$. Moreover, it is clear that $\C_\S$ is finitary.

Let $\S$ be a sentential logic. The \textit{Frege relation} of $\S$, in symbols $\Lambda(\S)$, is the interderivability relation, that is, $(\ph,\psi)\in\Lambda(\S)$ if and only if $\ph\vS\psi$ and $\psi\vS\ph$. The Frege relation of a sentential logic is an equivalence relation but it is not necessarily a congruence on $\Fm$. A sentential logic $\S$ is said to be \textit{selfextensional} (or $\S$ has \textit{the congruence property}) if the Frege relation $\Lambda(\S)$ is a congruence on $\Fm$.

Let $A$ be an algebra of the same similarity type as $\S$. A subset $F\subseteq A$ is said to be an $\S$-\textit{filter} of $A$ if and only if for any $\Gamma\cup\{\ph\}\subseteq\Fm$ and any interpretation $h\in\Hom(\Fm,A)$,
\[
\text{if } \Gamma\vS\ph \text{ and } h[\Gamma]\subseteq F, \text{ then } h(\ph)\in F.
\]
The set of all $\S$-filters on a given algebra $A$ is denoted by $\Fi_\S(A)$; this set is an algebraic closure system. The associated closure operator will be denoted by $\Fig_\S^A$.

Let $\Ll$ be a fixed but arbitrary algebraic language. A \textit{generalized matrix}, \textit{g-matrix} for short, of similarity type $\Ll$ is a pair $\langle A,\Cs\rangle$ where $A$ is an algebra of type $\Ll$ and $\Cs$ is an algebraic closure system on $A$. We denote by $\C$ the closure operator associated with $\Cs$  and we will often identify the g-matrix $\langle A,\Cs\rangle$ with the pair $\langle A,\C\rangle$. Notice that the closure operator $\C$ is finitary, i.e., for all $X\cup\{a\}\subseteq A$, $a\in\C(X)$ implies that there is a finite $X_0\subseteq X$ such that $a\in\C(X_0)$. 

The reader should keep in mind that all logics and g-matrices considered in this paper are finitary and thus some general results of AAL are restricted to this assumptions.

One of the most interesting aspects of g-matrices is that they can be used in a completely natural way both as models of sentential logics and as models of Gentzen systems. This double function of g-matrices allows relating the algebraic theory of sentential logics to the Gentzen systems. We address the interested reader on these topics  to \cite{FoJa09} and \cite{FoJaPi03}.

An important example of g-matrix is given by a sentential logic $\S$. If $\S$ is a sentential logic, then $\langle\Fm,\Th(\S)\rangle$ is a g-matrix.

\begin{definition}
A g-matrix $\langle A,\C\rangle$ is said to be a \textit{g-model} of a sentential logic $\S$ when for all $\Gamma\cup\{\ph\}\subseteq\Fm$, if $\Gamma\vdash_\S\ph$ then $h(\ph)\in\C(h[\Gamma])$ for all $h\in\Hom(\Fm,A)$. Let us denote the class of all g-models of a sentential logic $\S$ by $\GMod(\S)$.
\end{definition}

The logical concept of Frege relation is transferred to the setting of g-matrices. The \textit{Frege relation} of a g-matrix $\langle A,\C\rangle$ is defined by:
\[
(a,b)\in\Lambda_A(\C)\iff \C(a)=\C(b)
\]
for every $a,b\in A$. The \textit{Tarski congruence} of a g-matrix $\langle A,\C\rangle$ is the largest congruence below the Frege relation of the g-matrix. We denote the Tarski congruence of $\langle A,\C\rangle$ by $\widetilde{\Omega}_A(\C)$. A g-matrix is said to be \textit{reduced} when its Tarski congruence is the identity relation. Let us denote by $\GMod^*(\S)$ the class of all reduced g-models of a sentential logic $\S$. 

We can now introduce the class of algebras that is considered in AAL as the natural algebraic counterpart of a sentential logic, see \cite{FoJaPi03,Cz03}. 

\begin{definition}
The \textit{canonical class of algebras} associated with a sentential logic $\S$ (it is also called the \textit{algebraic counterpart} of $\S$) is the class of the algebraic reducts of the reduced g-models of $\S$; it is denoted by $\Alg(\S)$. That is, 
\begin{equation*}
\begin{split}
\Alg(\S):&=\Alg(\GMod^*(\S))\\
&=\{A: \langle A,\C\rangle\in\GMod^*(\S) \ \text{for some finitary closure operator } \C\}.
\end{split}
\end{equation*}
Moreover, another important class of algebras associated with a sentential logic $\S$ is $\mathrm{K}_\S:=\mathbb{V}(\Fm/\widetilde{\Omega}(\S))$, the variety generated by the algebra $\Fm/\widetilde{\Omega}(\S)$. This variety is called \textit{the intrinsic variety of} $\S$.
\end{definition}

\begin{lemma}[{\cite[Proposition 2.26]{FoJa09}}]\label{lem:AlgS<K_S}
Let $\S$ be a sentential logic. Then, the intrinsic variety of $\S$ is the variety generated by the class $\Alg(\S)$ and hence we have $\Alg(\S)\subseteq\V(\Alg(\S))=\mathrm{K}_\S$.
\end{lemma}

\begin{definition}\label{def:charac fully selfextensional}
A sentential logic $\S$ is said to be \textit{fully selfextensional} (or \textit{congruential}) if for every $A\in\Alg(\S)$, the Frege relation of the g-matrix $\langle A,\Fi_\S(A)\rangle$ is the identity relation.
\end{definition}

\subsection{Distributive nearlattices}\label{subsec:NL}

Nearlattices and distributive nearlattices were studied by several authors \cite{Hi80,ChaKo08,ChaHaKu07,ArKi11,ChaHa06,ChaKo06,CeCa14,CeCa16,Go17,GoLa17}.

\begin{definition}[\cite{ArKi11}]\label{def: DN-alg}
An algebra $\langle A,m\rangle$ of type (3) is called a \textit{nearlattice} if the following identities hold:
\begin{enumerate}[({P}1)]\label{equa:nearlattice}
		\item $m(x,y,x)=x$,
		\item $m(m(x,y,z),m(y,m(u,x,z),z),w)=m(w,w,m(y,m(x,u,z),z))$.
\end{enumerate}
\end{definition}

\begin{theorem}[\cite{ChaKo08}]\label{theo:characterisation DN}
Let $\langle A,m\rangle$ be an algebra of type (3) and let $\vee$ be the binary operation on $A$ defined by $x\vee y:= m(x,x,y)$. Then, $\langle A,m\rangle$ is a nearlattice if and only if $\langle A,\vee\rangle$ is a join-semilattice where for every $a\in A$, the principal upset $[a)=\{x\in A: a\leq x\}$ is a lattice with respect to the order $\leq$ induced by $\vee$. Moreover, for all $x,y,a\in A$, $m(x,y,a)=(x\vee a)\wedge_a(y\vee a)$ where $\wedge_a$ denotes the meet in $[a)$.
\end{theorem}

Let $\langle A,m\rangle$ be a nearlattice. Notice that the partial order $\leq$ on $A$ is determined by $\vee$, i.e., $x\leq y$ if and only if $y=x\vee y=m(x,x,y)$. Moreover, for every element $a\in A$ we have 
\[
x\wedge_ay=m(x,y,a)
\]
for all $x,y\in[a)$.  It should be noted that the meet $x\wedge y$ exists in $A$ if and only if $\{x,y\}$ has a lower bound in $A$. Thus, the meet of $x$ and $y$ in $[a)$ coincides with their meet in $A$ for all $x,y\in[a)$, i.e., $x\wedge_a y=x\wedge y$; for instance, we  have 
\[
m(a,b,c)=(a\vee c)\wedge_c(b\vee c)=(a\vee c)\wedge(b\vee c),
\]
for all $a,b,c\in A$. This should be kept in mind since we will use it without mention.

\begin{definition}\label{def:distr nearlattice}
A nearlattice $\langle A,m\rangle$ is said to be \textit{distributive} if and only if  it satisfies either of the following two equivalent identities:
\begin{enumerate}
	\item[(P3)] $m(x,m(y,y,z),w)=m(m(x,y,w),m(x,y,w),m(x,z,w))$;
	\item[(P4)] $m(x,x,m(y,z,w))=m(m(x,x,y),m(x,x,z),w)$.
\end{enumerate}
\end{definition}

Let us denote by $\DN$ the variety of distributive nearlattices.

\begin{proposition}[{\cite[Theorem 4]{ChaKo08}}]
A nearlattice $\langle A,m\rangle$ is distributive if and only if for every $a\in A$, the lattice $\langle[a),\wedge_a,\vee\rangle$ is distributive.
\end{proposition}

\begin{example}
In Figure \ref{fig:example}, it is shown a distributive nearlattice. For instance, we have 
\[
m(u,w,y)=(u\vee y)\wedge_y(w\vee y)=y \quad \text{and} \quad m(u,w,b)=(u\vee b)\wedge_b(w\vee b)=y.
\]
That is, $m(u,w,y)=u\wedge_y w=u\wedge w=u\wedge_bw=m(u,w,b)$.
\end{example}

\begin{figure}
\begin{tikzpicture}[scale=.5,inner sep=.5mm]
\node[bull] (a) at (-2,0) [label=below:$a$] {};%
\node[bull] (b) at (0,0) [label=below:$b$] {};%
\node[bull] (c) at (2,0) [label=below:$c$] {};%
\node[bull] (x) at (-2,2) [label=left:$x$] {};%
\node[bull] (y) at (0,2) [label=left:$y$] {};%
\node[bull] (z) at (2,2) [label=right:$z$] {};%
\node[bull] (u) at (-2,4) [label=left:$u$] {};%
\node[bull] (v) at (0,4) [label=left:$v$] {};%
\node[bull] (w) at (2,4) [label=right:$w$] {};%
\node[bull] (top) at (0,6) [label=above:$1$] {};%

\draw (a) -- (x) -- (u)--(top)--(v)--(x);%
\draw (b)--(y)--(u);%
\draw (c)--(z)--(w)--(top)--(v)--(z);%
\draw (y)--(w);%
\end{tikzpicture}
\caption{A distributive nearlattice}
\label{fig:example}
\end{figure}
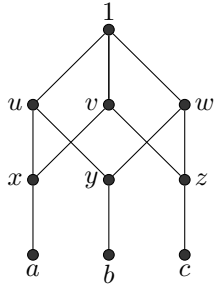

\begin{remark}\label{rem:generalisation of DN}
Recall that a \textit{Tarski algebra} (also called \textit{implication algebra}) \cite{Abb67,Abb67a} can be defined as a binary algebra $\langle A,\to\rangle$ satisfying some identities and, equivalently, it can be defined as a join-semilattice $\langle A,\vee\rangle$ such that for every $a\in A$, the upset $[a)$ is a Boolean algebra with respect to the order induced by $\vee$. Thus, we can noticed that the concept of distributive nearlattice is a natural generalisation of the notion of Tarski algebra.
\end{remark}

\begin{definition}
Let $\langle A,m\rangle$ be a distributive nearlattice. A nonempty subset  $F\subseteq A$ is said to be a \textit{filter} of $A$ if (i) $x\in F$ and $x\leq y$ implies $y\in F$, and (ii) if $x,y\in F$ and $x\wedge y$ exists in $A$, then $x\wedge y\in F$.
\end{definition}

Let us denote by $\Fi(A)$ the collection of all filters of a distributive nearlattice $A$. It is easy to check that for every distributive nearlattice $A$ the intersection of any collection of filters is either a filter or an empty set. So, for every nonempty $X\subseteq A$, there exists the least filter  containing $X$; it is denoted by $\Fig_A(X)$. If $X=\{a_1,\dots,a_n\}$, then we write $\Fig_A(a_1,\dots,a_n)$ instead $\Fig_A(\{a_1,\dots,a_n\})$; moreover, it is easy to check that $\Fig_A(a)=[a)$. There is a useful characterisation of the generated filter $\Fig_A(X)$. Let $X\subseteq A$ be nonempty. Then,
\[
\Fig_A(X)=\{a\in A:  a=b_1\wedge\dots\wedge b_k \text{ for some } b_1,\dots,b_k\in[X)\}.
\]
where $[X)=\{a\in A: a\geq x \text{ for some } x\in X\}$, see \cite{CoHi78}.

\begin{proposition}[\cite{Go17}]\label{prop:charac of filter}
Let $A$ be a nearlattice and $F\subseteq A$ be nonempty. Then, the following conditions are equivalent:
\begin{enumerate}
	\item[\normalfont{(1)}] $F\in\Fi(A)$;
	\item[\normalfont{(2)}] if $a,b\in F$, then $m(a,b,c)\in F$ for all $c\in A$.
\end{enumerate}
\end{proposition}

Now we introduce the following definition that will be useful for what follows.

\begin{definition}
Let $\langle A,m\rangle$ be an algebra of type (3). For each integer $n\geq0$, we define inductively, for all $a_0,\dots,a_n,b\in A$, an element $m^{n}(a_0,\dots,a_n,b)$ as follows:
\begin{itemize}
	\item $m^0(a_0,b):=m(a_0,a_0,b)$ and
	\item for $n\geq1$, $m^{n}(a_0,\dots,a_n,b):=m(m^{n-1}(a_0,\dots,a_{n-1},b),a_n,b)$.
\end{itemize}
\end{definition}

In particular, for distributive nearlattice $A$, we get $m^0(a_0,b)=a_0\vee b$ and $m^1(a_0,a_1,b)=m(a_0,a_1,b)$. Indeed,
\begin{multline*}
m^1(a_0,a_1,b)=m(m^0(a_0,b),a_1,b)=m(a_0\vee b,a_1,b)\\
=(a_0\vee b\vee b)\wedge_b(a_1\vee b)=(a_0\vee b)\wedge_b(a_1\vee b)=m(a_0,a_1,b).
\end{multline*}
The proofs of the following two propositions can be found in \cite{Go17}.

\begin{proposition}\label{prop:hom for m^n}
Let $\langle A,m_A\rangle$ and $\langle B,m_B\rangle$ be algebras of type (3) and $h\in\Hom(A,B)$. Then, we have $h(m_A^{n}(a_0,\dots,a_n,b))=m_B^{n}(h(a_0),\dots,h(a_n),h(b))$ for all $a_0,\dots,a_n,b\in A$.
\end{proposition}

\begin{proposition}\label{prop:properties of m^n}
Let $\langle A,m\rangle$ be a distributive nearlattice and $a_0,\dots,a_n,a_{n+1}$, $a,b\in A$. Then:
\begin{enumerate}[{\normalfont (1)}]
	\item $m^{n}(a_0,\dots,a_n,b)=(a_0\vee b)\wedge_b\dots\wedge_b(a_n\vee b)$;
	\item $b\leq m^{n}(a_0,\dots,a_n,b)$;
	\item if $a\leq a_i$ for all $i\in\{0,1,\dots,n\}$, then $a\leq m^{n}(a_0,\dots,a_n,b)$;
	\item $m^{n+1}(a_0,\dots,a_{n+1},b)\leq m^{n}(a_0,\dots,a_n,b)$;
	\item $m^{n}(a_0,\dots,a_n,b)=m^{n}(a_{\sigma(0)},\dots,a_{\sigma(n)},b)$, for every permutation $\sigma$ of $\{0,1,\dots,n\}$;
	\item if $a\in\Fig_A(a_0,\dots,a_n)$, then $a\in\Fig_A(m^{n}(a_0,\dots,a_n,a))$;
	\item $a\in\Fig_A(a_0,\dots,a_n)$ if and only if $a=m^{n}(a_0,\dots,a_n,a)$.
\end{enumerate}
\end{proposition}

\section{Logics with a distributive nearlatice term}\label{sec:DN-based logic}

For a ternary term $m$ of an algebraic language $\Ll$, we will consider the binary term $\vee$ defined by $x\vee y:=m(x,x,y)$. We also define, for every integer $n\geq 0$ and variables $x_0,\dots,x_n,x$, the formula $m^{n}(x_0,\dots,x_n,x)$ as follows:
\begin{itemize}
	\item $m^0(x_0,x):=m(x_0,x_0,x)$
	\item for $n\geq1$, $m^{n}(x_0,\dots,x_n,x):=m(m^{n-1}(x_0,\dots,x_{n-1},x),x_{n},x)$.
\end{itemize}

\begin{definition}\label{def:dn-term}
Let $\S$ be a sentential logic over an algebraic language $\Ll$. A ternary term $m$ of $\Ll$ is said to be a \textit{distributive nearlattice term} ($\dn$-\textit{term}) of $\S$ if and only if $\S$ satisfies the following properties:
\begin{enumerate}[({A}1)]
	\item $\ph\vee\psi\vdash_\S\chi$ if and only if $\ph\vdash_\S\chi$ and $\psi\vdash_\S\chi$;
	\item $m(\ph,\psi,\chi)\vdash_\S\ph\vee\chi$ and $m(\ph,\psi,\chi)\vdash_\S\psi\vee\chi$;
	\item $\ph\vee\chi,\psi\vee\chi\vdash_\S m(\ph,\psi,\chi)$;
	\item if $\ph_0,\dots,\ph_n\vdash_\S\ph$, then $m^{n}(\ph_0,\dots,\ph_n,\ph)\vdash_\S\ph$.
\end{enumerate}
\end{definition}

Property (A1) is known in the literature as the \textit{weak proof by cases property} \cite{WaCi08,CiNo11}. Moreover, (A1) implies that the following properties hold: $\ph\vee\ph\vdash_\S\ph$, $\ph\vdash_\S\ph\vee\psi$, $\ph\vee\psi\vdash_\S\psi\vee\ph$ and $(\ph\vee\psi)\vee\chi\vdash_\S\ph\vee(\psi\vee\chi)$. Then, the binary term $\vee$ is a (weak) \textit{disjunction} for the logic $\S$, and thus $\S$ is a (weak) \textit{disjunctive logic}, see \cite{Cz01,WaCi08,FoJa09,CiNo11}.

\begin{proposition}\label{prop:consequence from (A1)-(A4)}
If $m$ is a $\dn$-term of a sentential logic $\S$, then the following properties hold:
\begin{enumerate}[{\normalfont (1)}]
	\item $m^{n}(\ph_0,\dots,\ph_n,\psi)\vdash_\S\ph_i\vee\psi$, for all $i\in\{0,1,\dots,n\}$;
	\item $\ph_0\vee\psi,\dots,\ph_n\vee\psi\vdash_\S m^{n}(\ph_0,\dots,\ph_n,\psi)$;
	\item  if $m^{n}(\ph_0,\dots,\ph_n,\ph)\vdash_\S\ph$, then $\ph_0,\dots,\ph_n\vdash_\S\ph$;
	\item $\ph\vdash_\S m^{n}(\ph_0,\dots,\ph_n,\ph)$.
\end{enumerate}
\end{proposition}

\begin{proof}
Properties (1) and (2) can be proved by induction on $n$ using property (A1), and from properties (A2) and (A3), respectively. Properties (3) and (4) are consequences of (A1) and (2).
\end{proof}

\begin{proposition}
Let $\S$ be a sentential logic. If $m$ and $m'$ are $\dn$-terms of $\S$, then $m(\ph,\psi,\chi)\dashv\vdash_\S m'(\ph,\psi,\chi)$, for all $\ph,\psi,\chi\in\Fm$.
\end{proposition}

\begin{proof}
Let $\ph,\psi,\chi\in\Fm$. Since $\ph\vdash_\S\ph\vee'\psi$ and $\psi\vdash_\S\ph\vee'\psi$, it follows by (A1) that $\ph\vee\psi\vdash_\S\ph\vee'\psi$. Similarly, we have $\ph\vee'\psi\vdash_\S\ph\vee\psi$. Hence $\ph\vee\psi\dashv\vdash_\S\ph\vee'\psi$. Now, from this and by (A2) and (A3), we obtain
\[
m(\ph,\psi,\chi)\dashv\vdash_\S\{\ph\vee\chi,\psi\vee\chi\}\dashv\vdash_\S\{\ph\vee'\chi,\psi\vee'\chi\}\dashv\vdash_\S m'(\ph,\psi,\chi).
\qedhere
\]
\end{proof}

\begin{definition}\label{def:DN-class}
A class of algebras $\K$ of a given similarity type $\Ll$ is called \textit{distributive nearlattice-based} ($\dn$-\textit{based} for short) if there is a ternary term $m$ of $\Ll$ such that the distributive nearlattice equations (P1)-(P3) (Definitions \ref{def: DN-alg} and \ref{def:distr nearlattice}) hold in $\K$. In this case, we will also say that $\K$ is a $\dn$-\textit{class relative to} $m$ and when there is not danger of confusion, we simply say that $\K$ is a $\dn$-class.
\end{definition}

Notice that if $\K$ is a $\dn$-class relative to $m$, then for every algebra $A\in\K$ the $\{m\}$-reduct $\langle A,m^A\rangle$ is a distributive nearlattice. Moreover, we have that the variety $\V(\K)$ generated by a $\dn$-class $\K$ is also a $\dn$-class.

\begin{definition}\label{def:DN-based logic}
A sentential logic $\S$ of type $\Ll$ is said to be \textit{distributive near\-la\-ttice-based} ($\dn$-\textit{based} for short) if and only if there is a ternary term $m$ and a DN-class $\K$ of algebras  of type $\Ll$ such that the following condition holds for every $n\geq0$ and for all formulas $\ph_0,\dots,\ph_n,\ph$:
\begin{equation}\label{equa:def DN-based}
\begin{split}
\ph_0,\dots,\ph_n\vdash_\S\ph\iff&(\forall A\in\K)(\forall h\in\Hom(\Fm,A))\\
&\qquad m^{n}(h(\ph_0),\dots,h(\ph_n),h(\ph))\leq h(\ph).
\end{split}
\end{equation}
We will say that $\S$ is $\dn$-\textit{based relative to} $m$ \textit{and} $\K$.
\end{definition}

It should be noted, by property (5) of Proposition \ref{prop:properties of m^n}, that condition \eqref{equa:def DN-based} is independent of the order in which the formulas $\ph_0,\dots,\ph_n$ are taken.

Let $\S$ be a $\dn$-based logic relative to $m$ and $\K$. It is easy to check that for every formulas $\ph$ and $\psi$, 
\[
\ph\vdash_\S\psi \iff (\forall A\in\K)(\forall h\in\Hom(\Fm,A))(h(\ph) \leq h(\psi)).
\]
Then, we obtain that for all $\ph,\psi\in\Fm$
\begin{equation}\label{equa:aux 2}
\ph\dashv\vdash_\S\psi\iff \K\models\ph\approx\psi\iff\V(\K)\models\ph\approx\psi.
\end{equation}
Noticed that \eqref{equa:aux 2} is independent of the term $m$. 

From \eqref{equa:def DN-based} and by property (2) of Proposition \ref{prop:properties of m^n}, we have
\begin{equation*}
\begin{split}
\ph_0,\dots,\ph_n\vdash_\S\ph\iff 
&(\forall A\in\K)(\forall h\in\Hom(\Fm,A)\\
&\qquad m^{n}(h(\ph_0),\dots,h(\ph_n),h(\ph))= h(\ph)\\
\iff&
\K\models m^{n}(\ph_0,\dots,\ph_n,\ph)\approx\ph\\
\iff&
\V(\K)\models m^{n}(\ph_0,\dots,\ph_n,\ph)\approx\ph.
\end{split}
\end{equation*}
Hence, $\S$ is also $\dn$-based relative to the variety $\V(\K)$ generated by $\K$. Moreover, by \eqref{equa:aux 2}, we can see that the variety to which $\S$ is $\dn$-based is unique. So, let us denote the only variety relative to which $\S$ is $\dn$-based by $\V(\S)$.

\begin{proposition}\label{prop:dn-based to dn-term}
Let $\S$ be a $\dn$-based logic relative to $m$. Then, $m$ is a $\dn$-term of $\S$.
\end{proposition}

\begin{proof}
We have that $\S$ is $\dn$-based relative to $\V(\S)$ and the ternary term $m$. Property (A1) is a consequence of the fact that for every $A\in\V(\S)$, the $\{\vee\}$-reduct $\langle A,\vee^A\rangle$ is a join-semilattice. Property (A2) holds because for every $A\in\V(\S)$ and all $a,b,c\in A$, we have $m^A(a,b,c)=(a\vee c)\wedge(b\vee c)\leq a\vee c,b\vee c$. In order to prove (A3), let $\ph,\psi,\chi\in\Fm$. Let $A\in\V(\S)$ and $h\in\Hom(\Fm,A)$. Assume that $h(\ph)=a$, $h(\psi)=b$ and $h(\chi)=c$. So, we need to show that $m(a\vee c,b\vee c,m(a,b,c))\leq m(a,b,c)$. Now, by condition (P4), we have
\begin{multline*}
m(a\vee c,b\vee c,m(a,b,c))=c\vee m(a,b,m(a,b,c))=m(a,b,m(a,b,c))\\
=(a\vee m(a,b,c))\wedge(b\vee m(a,b,c))=(a\vee c)\wedge(b\vee c)=m(a,b,c).
\end{multline*}
Hence (A3) holds. Lastly, property (A4) is an immediate consequence by \eqref{equa:def DN-based}.
\end{proof}

By the previous proposition and from \eqref{equa:aux 2}, we obtain that if $\S$ is a $\dn$-based logic relative to ternary terms $m$ and $m'$, then for every $A\in\V(\S)$, we have that the ternary operations $m^A$ and $m'^A$ coincide. Hence, we can say simply that a logic $\S$ is $\dn$-based.

\begin{proposition}\label{prop:dn-based-->selfextensional}
Let $\S$ be a $\dn$-based logic. Then $\S$ is selfextensional and $\V(\S)=\K_\S$.
\end{proposition}

\begin{proof}
By definition of the Frege relation and  from \eqref{equa:aux 2}, we have $(\ph,\psi)\in\Lambda(\S)\iff\V(\S)\models\ph\approx\psi$. Hence, we obtain that $\Lambda(\S)$ is a congruence on $\Fm$. Therefore $\S$ is selfextensional. Now, since $\S$ selfextensional, it follows that $\ph\dashv\vdash_\S\psi\iff\K_\S\models\ph\approx\psi$. Then, by \eqref{equa:aux 2} again, we obtain that $\V(\S)\models\ph\approx\psi\iff\ph\dashv\vdash_\S\psi\iff\K_\S\models\ph\approx\psi$. Therefore $\V(\S)=\K_\S$.
\end{proof}

Now we are ready to show one of the main results of this paper.

\begin{theorem}\label{theo:dn-term iff dn-based}
Let $\S$ be a sentential logic. Then, $\S$ is a selfextensional logic with a $\dn$-term $m$ if and only if $\S$ is a $\dn$-based logic relative to $m$.
\end{theorem}

\begin{proof}
The implication from right to left is a consequence from Propositions \ref{prop:dn-based to dn-term} and \ref{prop:dn-based-->selfextensional}. Now we assume that $\S$ is selfextensional and $m$ is a $\dn$-term of $\S$. First, since $\S$ is selfextensional, it follows that $\Lambda(\S)$ is a congruence on $\Fm$ and hence, we can consider the quotient algebra $\Fm^*:=\Fm/\Lambda(\S)$. Let us show that $\langle\Fm^*,m^*\rangle$, with $m^*(\overline{\ph} ,\overline{\psi},\overline{\chi}):=\overline{m(\ph,\psi,\chi)}$ ($\overline{\ph}$ denotes the equivalent class of $\ph$ in $\Fm/\Lambda(\S)$), is a distributive nearlattice. By (A1), $\S$ satisfies the following properties: $\ph\vee\ph\vdash_\S\ph$, $\ph\vdash_\S\ph\vee\psi$, $\ph\vee\psi\vdash_\S\psi\vee\ph$ and $\ph\vee(\psi\vee\chi)\dashv\vdash_\S(\ph\vee\psi)\vee\chi$. Thus, it is easy to check that $\langle\Fm^*,\vee^*\rangle$, with $\overline{\ph}\vee^*\overline{\psi}:=m^*(\overline{\ph},\overline{\ph},\overline{\psi})$, is a join-semilattice. Let $\chi\in\Fm$. We prove that $[\overline{\chi})=\{\overline{\ph}\in\Fm^*:\overline{\chi}\leq\overline{\ph}\}=\{\overline{\ph}\in\Fm^*:\chi\vdash_\S\ph\}$ is a distributive lattice. In order to prove that $[\overline{\chi})$ is a lattice, we need only to show that there exists the meet in $[\overline{\chi})$. Let $\overline{\ph},\overline{\psi}\in[\overline{\chi})$. So $\chi\vdash_\S\ph,\psi$. Let us prove that $m^*(\overline{\ph},\overline{\psi},\overline{\chi})$ is the meet of $\overline{\ph}$ and $\overline{\psi}$ in $[\overline{\chi})$. By (A1) and (A2) we have $m(\ph,\psi,\chi)\vdash_\S\ph\vee\chi\vdash_\S\ph$ and $m(\ph,\psi,\chi)\vdash_\S\psi\vee\chi\vdash_\S\psi$. Thus $m^*(\overline{\ph},\overline{\psi},\overline{\chi})\leq\overline{\ph},\overline{\psi}$. Let $\overline{\gamma}\in[\overline{\chi})$ be such that $\overline{\gamma}\leq\overline{\ph},\overline{\psi}$. So $\chi\vdash_\S\gamma$ and $\gamma\vdash_\S\ph,\psi$. Then $\gamma\vdash_\S\ph\vee\chi,\psi\vee\chi$. By (A3) we obtain that $\gamma\vdash_\S m(\ph,\psi,\chi)$, that is, $\overline{\gamma}\leq m^*(\overline{\ph},\overline{\psi},\overline{\chi})$. Hence $m^*(\overline{\ph},\overline{\psi},\overline{\chi})=\overline{\ph}\wedge_{\overline{\chi}}\overline{\psi}$. Then, by Theorem \ref{theo:characterisation DN}, we conclude that $\langle\Fm^*,m^*\rangle$ is a nearlattice. Now we show that condition (P4) holds in $\langle\Fm^*,m^*\rangle$. Let $\ph,\psi,\gamma,\chi\in\Fm$. Since $\langle\Fm^*,m^*\rangle$ is a nearlattice, it follows that $\overline{\ph}\vee^*m^*(\overline{\psi},\overline{\gamma},\overline{\chi})\leq m^*(\overline{\ph}\vee^*\overline{\psi},\overline{\ph}\vee^*\overline{\gamma},\overline{\chi})$. In order to prove the inverse inequality, we need to show that $m(\ph\vee\psi,\ph\vee\gamma,\chi)\vdash_\S\ph\vee m(\psi,\gamma,\chi)$. By (A1), we have $\psi,\gamma\vdash_\S\psi\vee\chi,\gamma\vee\chi$, and from (A3) we obtain that $\psi,\gamma\vdash_\S m(\psi,\gamma,\chi)$. Thus $\psi,\gamma\vdash_\S\ph\vee m(\psi,\gamma,\chi)$. By (A4), it follows that
\begin{equation}\label{equa:aux 3}
m(\psi,\gamma,\ph\vee m(\psi,\gamma,\chi))\vdash_\S\ph\vee m(\psi,\gamma,\chi).
\end{equation} 
By (A1) and (A3), we can deduce $\ph\vee\psi\vee\chi\vdash_\S\ph\vee\psi\vee m(\psi,\gamma,\chi)$ and $\ph\vee\gamma\vee\chi\vdash_\S\ph \vee\gamma\vee m(\psi,\gamma,\chi)$. Then, by (A1)-(A3) and \eqref{equa:aux 3}, we have
\begin{align*}
m(\ph\vee\psi,\ph\vee\gamma,\chi)
&\vdash_\S
\ph\vee\psi\vee\chi, \ \ph\vee\gamma\vee\chi\\ 
&\vdash_\S
\ph\vee\psi\vee m(\psi,\gamma,\chi), \ \ph\vee\gamma\vee m(\psi,\gamma,\chi)\\
&\vdash_\S
\psi\vee\ph\vee m(\psi,\gamma,\chi), \ \gamma\vee\ph\vee m(\psi,\gamma,\chi)\\
&\vdash_\S
m(\psi,\gamma,\ph\vee m(\psi,\gamma,\chi))\\
&\vdash_\S
\ph\vee m(\psi,\gamma,\chi).
\end{align*}
Hence, we have proved that $\langle\Fm^*,m^*\rangle$ is a distributive nearlattice. Finally, we prove that $\S$ is $\dn$-based relative to $\{\Fm^*\}$ and $m$. Let $\ph_0,\dots,\ph_n,\ph\in\Fm$. From property (A4), (3) of Proposition \ref{prop:consequence from (A1)-(A4)}, (2) of Proposition \ref{prop:properties of m^n} and since $\S$ is selfextensional, it follows that
\begin{multline*}
\ph_0,\dots,\ph_n\vdash_\S\ph
\iff 
m^{n}(\ph_0,\dots,\ph_n,\ph)\vdash_\S\ph\\
\iff
\overline{m^n(\ph_0,\dots,\ph_n,\ph)}\leq\overline{\ph}
\iff
\overline{m^n(\ph_0,\dots,\ph_n,\ph)}=\overline{\ph}\\
\iff
\Fm^*\models m^{n}(\ph_0,\dots,\ph_n,\ph)\approx\ph\\
\iff
(\forall h\in\Hom(\Fm,\Fm^*))(m^{n}(h\ph_0,\dots,h\ph_n,h\ph)=h\ph)\\
\iff
(\forall h\in\Hom(\Fm,\Fm^*))(m^{n}(h\ph_0,\dots,h\ph_n,h\ph)\leq h\ph).
\end{multline*}
This completes the proof.
\end{proof}

Our next aim is to prove that every selfextensional logic $\S$ with a $\dn$-term is  fully selfextensional and the class $\Alg(\S)$ is a variety. Notice, by the previous theorem and Proposition \ref{prop:dn-based-->selfextensional}, that for every selfextensional logic $\S$ with a $\dn$-term $m$ the $m$-reducts of the algebras of its intrinsic variety $\K_\S$ are distributive nearlattices and $\S$ is $\dn$-based relative to $\K_\S$.

\begin{proposition}\label{prop:Fi_S(A)=Fi(A)}
Let $\S$ be a $\dn$-based logic relative to $m$. Then, for every algebra $A\in\K_\S$, the nonempty $\S$-filters of $A$ are exactly the filters of the $\{m\}$-reduct distributive nearlattice $\langle A,m^A\rangle$, i.e., $\Fi_\S(A)\setminus\{\emptyset\}=\Fi(A)$.
\end{proposition}

\begin{proof}
Let $A\in\K_\S$. Let $F\in\Fi(A)$. Let $\ph_0,\dots,\ph_n,\ph\in\Fm$ be such that $\ph_0,\dots,\ph_n\vdash_\S\ph$ and let $h\in\Hom(\Fm,A)$ be such that $h(\ph_i)\in F$ for all $i=0,1,\dots,n$. By \eqref{equa:def DN-based}, we have $m^{n}(h(\ph_0),\dots,h(\ph_n),h(\ph))\leq h(\ph)$. Since $F$ is a filter of the nearlattice $A$ and $h(\ph_0),\dots,h(\ph_n)\in F$, it follows by Proposition \ref{prop:charac of filter} that $m^{n}(h(\ph_0),\dots,h(\ph_n),h(\ph))\in F$. Then $h(\ph)\in F$. Hence $F\in\Fi_\S(A)$. Conversely, let now $F\in\Fi_\S(A)$ be nonempty.  Let $a,b\in F$ and $c\in A$. By (A1) and (A3) we have, for variables $x$, $y$ and $z$, that $x,y\vdash_\S \{x\vee z,y\vee z\}\vdash_\S m(x,y,z)$. By taking $h\in\Hom(\Fm,A)$ such that $h(x)=a$, $h(y)=b$ and $h(z)=c$, we obtain that $h(x),h(y)\in F$ and hence $h(m(x,y,z))\in F$, i.e., $m(a,b,c)\in F$. Therefore, $F\in\Fi(A)$.
\end{proof}

\begin{theorem}\label{theo:AlgS variety}
Let $\S$ be a $\dn$-based logic. Then:
\begin{enumerate}[{\normalfont (1)}]
	\item $\Alg(\S)=\K_\S$;
	\item $\Alg(\S)$ is  a variety;
	\item $\S$ is $\dn$-based relative to $\Alg(\S)$.
\end{enumerate}
\end{theorem}

\begin{proof}
(1) We know by Lemma \ref{lem:AlgS<K_S} that $\Alg(\S)\subseteq\K_\S$. Let $A\in\K_\S$. From Proposition \ref{prop:Fi_S(A)=Fi(A)} we can easily deduce that the g-matrix $\langle A,\Fi_\S(A)\rangle$ is a reduced g-model of $\S$. Hence $A\in\Alg(\S)$. Properties (2) and (3) are immediate consequences of (1).
\end{proof}

\begin{corollary}
If $\S$ is a $\dn$-based logic, then $\S$ is fully selfextensional.
\end{corollary}

\begin{proof}
Recall the definition of fully selfextenisonality, see Definition \ref{def:charac fully selfextensional}. Let $A\in\Alg(\S)$. So $A\in\K_\S$. Then, by Proposition \ref{prop:Fi_S(A)=Fi(A)}, it is easy check that $\Lambda_A(\Fi_\S(A))=\Id_A$. Hence $\S$ is fully-selfextensional.
\end{proof}

We have characterised selfextensional logics with a $\dn$-term as those logics that are $\dn$-based concerning their canonical class of algebras. As happens in the setting of selfextensional logics with a conjunction \cite{Ja06}, two different sentential logics $\S$ and $\S'$ can be $\dn$-based relative to the same $\dn$-based variety $\K$. The unique possible case for this is when one of them has theorems and the other has not. Now we will see under what conditions the uniqueness can be obtained.

Let $\K$ be a $\dn$-based variety relative to a ternary term $m$. Let us define the sentential logic $\S_\K=\langle\Fm,\vdash_\K\rangle$\label{def:logic S_K} as follows: let $\ph_0,\dots,\ph_n,\ph\in\Fm$,
\begin{equation}\label{equa:def S_K-1}
\begin{split}
\ph_0,\dots,\ph_n\vdash_{\K}\ph
\iff
(\forall A\in\K)&(\forall h\in\Hom(\Fm,A))\\
& m^{n}(h\ph_0,\dots,h\ph_n,h\ph)\leq h\ph
\end{split}
\end{equation}
and
\begin{equation}\label{equa:def S_K-2}
\begin{split}
\emptyset\vdash_{\K}\ph
\iff
(\forall A\in\K)&(\forall h\in\Hom(\Fm,A))\\
& (\forall a\in A)(a\leq h\ph).
\end{split}
\end{equation}
Now, for every $\Gamma\subseteq\Fm$, $\Gamma\vdash_{\K}\ph$ if and only if there is a finite $\Gamma_0\subseteq\Gamma$ such that $\Gamma_0\vdash_{\K}\ph$. Notice that if $\S_\K$ has a theorem, then for every algebra $A\in\K$ the $\{m\}$-reduct nearlattice $\langle A,m^A\rangle$ has a greatest element. Moreover, since $\K$ is a variety, it follows that $\K$ is the unique $\dn$-based variety to which $\S_\K$ is $\dn$-based and hence, by Theorem \ref{theo:AlgS variety}, we have $\K_{\S_\K}=\Alg(\S_\K)=\K$\label{AlgS_K=K}.

A sentential logic $\S$ is said to be \textit{non-pseudo axiomatic} (\cite{LoSu58}) if for every formula $\ph$, $\ph$ is a theorem if and only if $\ph$ is derivable from every formula ($\psi\vdash_\S\ph$ for all formula $\psi$), or equivalently if the intersection of all its nonempty theories is the set of theorems. Notice that every sentential logic with theorems is non-pseudo axiomatic. The following proposition is an immediate consequence from \eqref{equa:def S_K-1} and \eqref{equa:def S_K-2}, and thus we omit its proof.

\begin{proposition}\label{prop:K_S_K=K}
Let $\Ll$ be an algebraic language and $m$ a ternary term of $\Ll$. If $\K$ is a $\dn$-based variety relative to $m$, then $\S_\K$ is the unique non-pseudo axiomatic sentential logic which is $\dn$-based relative to $\K$ and $m$; moreover $\K_{\S_\K}=\K$. If $\S$ is a $\dn$-based and non-pseudo axiomatic logic  relative to $m$, then $\S_{\K_\S}=\S$.
\end{proposition}

Hence, under the condition of non-pseudo axiomatic, we obtain the following kind of uniqueness for $\dn$-based logics: different non-pseudo axiomatic logics must be $\dn$-based relative to different $\dn$-based varieties.

Now, we show a bijective correspondence between the class of $\dn$-based and non-pseudo axiomatic logics and the class of subvarieties of the variety axiomatized by equations (P1)-(P3).

A sentential logic $\S'$ is said to be an \textit{extension} of a sentential logic $\S$ if and only if for every $\Gamma\cup\{\ph\}\subseteq\Fm$, $\Gamma\vdash_\S\ph$ implies $\Gamma\vdash_{\S'}\ph$.

\begin{lemma}\label{lem:Frege relation and extension}
Let $\S$ and $\S'$ be $\dn$-based and non-pseudo axiomatic logics. Then, $\Lambda(\S)\subseteq\Lambda(\S')$ if and only if $\S'$ is an extension of $\S$.
\end{lemma} 

\begin{proof}
It is immediate that if $\S'$ is an extension of $\S$, then $\Lambda(\S)\subseteq\Lambda(\S')$. So, we need to prove the implication from left to right. Assume that $\Lambda(\S)\subseteq\Lambda(\S')$. By property (A4) and properties (3) and (4) of Proposition \ref{prop:consequence from (A1)-(A4)}, it follows that
\begin{multline*}
\ph_0,\dots,\ph_n\vdash_\S\ph \iff 
m^{n}(\ph_0,\dots,\ph_n,\ph)\vdash_\S\ph \iff\\
m^{n}(\ph_0,\dots,\ph_n,\ph)\dashv\vdash_\S\ph \implies
m^{n}(\ph_0,\dots,\ph_n,\ph)\dashv\vdash_{\S'}\ph \iff\\
m^{n}(\ph_0,\dots,\ph_n,\ph)\vdash_{\S'}\ph \iff
\ph_0,\dots,\ph_n\vdash_{\S'}\ph.
\end{multline*}
Now, if $\emptyset\vdash_\S\ph$, then  $\psi\vdash_\S\ph$ for every formula $\psi$. By the above, we obtain that $\psi\vdash_{\S'}\ph$ for every formula $\psi$. Now, since $\S'$ is non-pseudo axiomatic, it follows that $\emptyset\vdash_{\S'}\ph$. Hence, we have proved that $\S'$ is an extension of $\S$.
\end{proof}

Let $\Ll$ be an algebraic language and $m$ a ternary term of $\Ll$. We set
\begin{itemize}
	\item $\mathbb{S}_m(\Ll):=\{\S: \S \text{ is a }  \text{non-pseudo-axiomatic logic over } \Ll \text{ and } \dn\text{-based}\\
\text{\hspace{18mm} relative to } m\}$ and
\item $\mathbb{K}_m(\Ll):=\{\K: \K \text{ is a subvariety of the variety over } \Ll \text{ axiomatized}\\
\text{\hspace{18mm} by the equations (P1)-(P3) with regard to } m\}$.
\end{itemize}
We consider  $\mathbb{S}_m(\Ll)$ ordered by the extension order, i.e., $\S\leq\S'$ if and only if $\S'$ is an extension of $\S$ and $\mathbb{K}_m(\Ll)$ ordered by the inclusion order. Now, we are in a position to establish and prove the announced result above.

\begin{theorem}\label{theo:S isom K}
Let $\Ll$ be an algebraic language and $m$ a ternary term of $\Ll$. Then, the map $\mathrm{F}\colon\mathbb{S}_m(\Ll)\to\mathbb{K}_m(\Ll)$ defined by $\mathrm{F}(\S)=\K_\S$, is a dual order isomorphism.
\end{theorem}

\begin{proof}
By Proposition \ref{prop:dn-based-->selfextensional}, we have that $\mathrm{F}$ is well defined, and by Proposition \ref{prop:K_S_K=K} we obtain that $\mathrm{F}$ is an onto map. Let $\S,\S'\in\mathbb{S}_m(\Ll)$. Then, by Lemma \ref{lem:Frege relation and extension} and using that $\S$ and $\S'$ are selfextensional, we have
\begin{align*}
\S\leq\S'
&\iff
\Lambda(\S)\subseteq\Lambda(\S')\\
&\iff
(\forall \ph,\psi\in\Fm)(\K_{\S}\models\ph\approx\psi\implies\K_{\S'}\models\ph\approx\psi)\\
&\iff
\K_{\S'}\subseteq\K_\S.
\end{align*}
Therefore, $\mathrm{F}$ is a dual order isomorphism.
\end{proof}

\section{Two examples}\label{sec:examples}

\subsection{The logic of distributive nearlattices}

In \cite{Go17}, it is defined a sentential logic $\S_{dn}$ through a Gentzen calculus, which can be considered as naturally associated with the variety of distributive nearlattices. There, the logic $\S_{dn}$ is denoted by $\S_\DN$, but here $\S_\DN$ has a specific definition, see \eqref{equa:def S_K-1} and \eqref{equa:def S_K-2}. Let us show that $\S_{dn}$ is the unique  non-pseudo axiomatic $\dn$-based logic relative to the variety of distributive nearlattices $\DN$. To this end, we need to introduce some basic notions of Gentzen calculus; we refer the reader to \cite{FoJa09} and \cite{Go17} for more information.

Let $\Fm$ be the algebra of formulas of a given algebraic similarity type $\Ll$. For our purpose, we will consider a \textit{sequent} of type $\Ll$ to be a pair $\langle\Gamma,\ph\rangle$ where $\Gamma$ is a (possible empty) finite set of formulas and $\ph$ is a formula. As usual, we write  $\Gamma\vtl\ph$ instead of $\langle\Gamma,\ph\rangle$. Let us denote by $\Seq(\Ll)$ the collection of all sequents. A \textit{Gentzen-style rule} is a pair $\langle X,\Gamma\vtl\ph\rangle$ where $X$ is a (possible empty) finite set of sequents and $\Gamma\vtl\ph$ is a sequent. As usual, we shall use the standard fraction notation for Gentzen-style rules:
\begin{equation}\label{equa:g-rule}
\AxiomC{$\Gamma_0\vtl\ph_0,\dots,\Gamma_{n-1}\vtl\ph_{n-1}$}
\UnaryInfC{$\Gamma\vtl\ph$}
\DisplayProof
\end{equation}
A \textit{substitution instance of a Gentzen-style rule} $\langle X,\Gamma\vtl\ph\rangle$ is a Gentzen-style rule of the form $\langle\sigma[X],\sigma[\Gamma]\vtl\sigma(\ph)\rangle$ for some substitution $\sigma\in\Hom(\Fm,\Fm)$. A \textit{Gentzen calculus} is a set of Gentzen-style rules. Given a Gentzen calculus \textbf{G}, the notion of a formal proof can be defined as usual. That is, a \textit{proof} in the Gentzen calculus \textbf{G} from a set of sequents $X$ is a finite sequence of sequents each one of whose elements is a substitution instance of a rule of \textbf{G} or a sequent in $X$ or is obtained by applying a substitution instance of a rule of \textbf{G} to previous elements in the sequence. A sequent $\Gamma\vtl\ph$ is \textit{derivable in} \textbf{G} \textit{from} a set of sequents $X$ if there is a proof in \textbf{G} from $X$ whose last sequent in the proof is $\Gamma\vtl\ph$. We express this writing $X\nc_{\textbf{G}}\Gamma\vtl\ph$.

\begin{definition}
A \textit{Gentzen system}  is a pair $\G=\langle\Fm,\nc_\G\rangle$ where $\nc_\G$ is a finitary closure operator on the set $\Seq(\Ll)$  that is substitution-invariant and which satisfies the following \textit{structural rules:} for every $\Gamma\cup\{\ph,\psi\}\subseteq\Fm$, 
\[\label{structural rules}
\AxiomC{$\emptyset$}
\LeftLabel{(Axiom)}
\UnaryInfC{$\ph\vtl\ph$}
\DisplayProof
\qquad
\AxiomC{$\Gamma\vtl\ph$}
\LeftLabel{(Weakening)}
\UnaryInfC{$\Gamma,\psi\vtl\ph$}
\DisplayProof
\qquad
\AxiomC{$\Gamma\vtl\ph$}
\AxiomC{$\Gamma,\ph\vtl\psi$}
\LeftLabel{(Cut)}
\BinaryInfC{$\Gamma\vtl\psi$}
\DisplayProof
\]
We say that a Gentzen system $\G=\langle\Fm,\nc_\G\rangle$ \textit{satisfies a Gentzen-style rule} of type \eqref{equa:g-rule} or \eqref{equa:g-rule} is a \textit{Gentzen-style rule of} $\G$ if $\Gamma_0\vtl\ph_0,\dots,\Gamma_{n-1}\vtl\ph_{n-1}\nc_\G\Gamma\vtl\ph$ and we say that a sequent $\Gamma\vtl\ph$ is a \textit{derivable sequent} of $\G$ when $\emptyset\nc_\G\Gamma\vtl\ph$.
\end{definition}

Let \textbf{G} be a Gentzen calculus with the structural rules of (Axiom), (Weakening) and (Cut). Hence, \textbf{G} defines in a standard way the Gentzen system $\G_{\textbf{G}}=\langle\Fm,\nc_{\textbf{G}}\rangle$ (see \cite{FoJa09,ReVe95}).

Now let $\Ll=\{m\}$ with $m$ a ternary connective.

\begin{definition}[{\cite[Definition 4.2]{Go17}}]\label{def:dn-logic}
Let $\G_{\DN}=\langle\Fm,\nc_{\DN}\rangle$ be the Gentzen system defined  by the following Gentzen-style rules: the structural rules (Axiom), (Weakening) and (Cut) and the following rules
\[
	\AxiomC{$\ph\vtl\chi$}
	\AxiomC{$\psi\vtl\chi$}
	\LeftLabel{$(\vee \ \vtl)$}
	\BinaryInfC{$\ph\vee\psi\vtl\chi$}
	\DisplayProof
\qquad
	\AxiomC{$\Gamma\vtl\ph$}
	\LeftLabel{$(\vtl \ \vee)$}
	\UnaryInfC{$\Gamma\vtl\ph\vee\psi$}
	\DisplayProof
	\ 
	\AxiomC{$\Gamma\vtl\ph$}
	\UnaryInfC{$\Gamma\vtl\ph\vee\psi$}
	\DisplayProof
\]
\vspace{1mm}
\[
	\AxiomC{}
	\LeftLabel{$(m \ \vtl)$}
	\UnaryInfC{$m(\ph,\psi,\chi)\vtl\ph\vee\chi$}
	\DisplayProof
	\ 
	\AxiomC{}
	\UnaryInfC{$m(\ph,\psi,\chi)\vtl\psi\vee\chi$}
	\DisplayProof
\]
\vspace{1mm}
\[
	\AxiomC{$\Gamma\vtl\ph\vee\chi$}
	\AxiomC{$\Gamma\vtl\psi\vee\chi$}
	\LeftLabel{$(\vtl \ m)$}
	\BinaryInfC{$\Gamma\vtl m(\ph,\psi,\chi)$}
	\DisplayProof
\qquad
	\AxiomC{$\ph_1,\dots,\ph_n\vtl\ph$}
	\LeftLabel{$(m^n \ \vtl)$}
	\UnaryInfC{$m^{n}(\ph_0,\dots,\ph_n,\ph)\vtl\ph$}
	\DisplayProof
\]

\noindent Now, the sentential $\S_{dn}=\langle\Fm,\vdash_{dn}\rangle$ is defined as follows: for all $\Gamma\cup\{\ph\}\subseteq\Fm$,
\[
\Gamma\vdash_{dn}\ph \iff \text{there is a finte } \Gamma_0\subseteq\Gamma \text{ such that } \nc_\DN\Gamma_0\vtl\ph.
\]
\end{definition}

It follows directly that the connective $m$ is a $\dn$-term of $\S_{dn}$. Indeed, property (A1) follows from rules $(\vee \ \vtl)$ and $(\vtl \ \vee)$; (A2) follows from $(m \ \vtl)$; property (A3) is a consequence of $(\vtl \ m)$; and lastly, (A4) follows from the rule $(m^n \ \vtl)$.

\begin{theorem}[{\cite[Theorem 4.15]{Go17}}]
The sentential logic $\S_{dn}$ has the following properties:
\begin{enumerate}[{\normalfont (1)}]
	\item $\Alg(\S_{dn})=\DN$;
	\item for all $\ph_0,\dots,\ph_n,\ph\in\Fm$,
	\[
	\ph_0,\dots,\ph_n\vdash_{dn}\ph \iff \DN\models m^{n}(\ph_0,\dots,\ph_n,\ph)\approx\ph.
	\]
	\item for every $A\in\Alg(\S_{dn})$, $\Fi_{\S_{dn}}(A)=\Fi(A)\cup\{\emptyset\}$;
	\item $\S_{dn}$ is fully selfextensional.
\end{enumerate}
\end{theorem}

Therefore, by condition (2) of the previous theorem, we have that the logic $\S_{dn}$ is $\dn$-based relative to the variety $\DN$. Moreover, since $\S_{dn}$ is non-pseudo axiomatic, it follows by  Proposition \ref{prop:K_S_K=K} that $\S_{dn}=\S_\DN$.

\subsection{Modal distributive nearlattices}

In \cite{Ca15}, a modal operator is defined on distributive nearlattices. There, the main tool to treat these operators was a topological duality for the category of distributive nearlattices. 

Let us consider the algebraic language $\Ll=\{m,\square,\top\}$ of type (3,1,0). 

\begin{definition}\label{def:modal}
An algebra $\langle A,m,\square,1\rangle$ is said to be a $\square$-\textit{modal distributive nearlattice} if $\langle A,m,1\rangle$ is a distributive nearlattice with a greatest element 1, and the following conditions hold:
\begin{enumerate}
	\item $\square 1=1$;
	\item for all $a,b\in A$ such that $a\wedge b$ there exists, $\square(a\wedge b)=\square a\wedge\square b$.
\end{enumerate}
\end{definition}

We denote by $\square\DN$ the collection of all $\square$-modal distributive nearlattices. Let us show that $\square\DN$ is a variety.

\begin{proposition}
Let $\langle A,m,\square,1\rangle$ be an algebra such that $\langle A,m,1\rangle$ is a distributive nearlattice with a greatest element $1$ and $\square 1=1$. Then, $\langle A,m,\square,1\rangle\in\square\DN$ if and only if the following identity holds in $\langle A,m,\square,1\rangle$:
\begin{equation}\label{equa:square}
\square m(x,y,z)=m(\square(x\vee z),\square(y\vee z),\square z).\tag{M}
\end{equation}
Therefore, $\square\DN$ is the variety defined by identities of distributive nearlattices with a greatest element, and the identities $\square 1=1$ and \eqref{equa:square}.
\end{proposition}

\begin{proof}
First assume that $\langle A,m,\square,1\rangle\in\square\DN$. Let $a,b,c\in A$. Since the operator $\square$ is order-preserving, it follows that $\square c\leq\square(a\vee c)$ and $\square c\leq\square(b\vee c)$. Then, we have
\begin{multline*}
\square m(a,b,c) = \square[(a\vee c)\wedge(b\vee c)] = \square(a\vee c)\wedge\square(b\vee c)\\
 = (\square(a\vee c)\vee\square c)\wedge(\square(b\vee c)\vee \square c) = m(\square(a\vee c),\square(b\vee c),\square c).
\end{multline*}
Hence, identity \eqref{equa:square} holds in $A$. Now, conversely, suppose that \eqref{equa:square} holds in $A$. We need to check that condition (2) in Definition \ref{def:modal} is satisfied. First, we show that $\square$ is order-preserving. Let $a,b\in A$ be such that $a\leq b$. So $b=b\vee a=m(b,b,a)$. Then, we have
\[
\square b=\square m(b,b,a)=m(\square(b\vee a),\square(b\vee a),\square a) = m(\square b,\square b,\square a)=\square b\vee \square a.
\]
Thus we obtain that $\square a\leq\square b$. Now let $a,b\in A$ be such that $a\wedge b$ there exists. Since $\square$ is order-preserving, we have $\square(a\wedge b)\leq\square a,\square b$. Hence
\begin{multline*}
\square(a\wedge b) = \square m(a,b,a\wedge b) = m(\square a,\square b,\square(a\wedge b))\\
 = (\square a\vee\square(a\wedge b))\wedge(\square b\vee\square(a\wedge b)) = \square a \wedge \square b.
\end{multline*}
Therefore $\langle A,m,\square,1\rangle$ is a $\square$-modal distributive nearlattice.
\end{proof}

Consider the $\dn$-based logic $\S_{\square\DN}$ on the algebraic language $\Ll=\{m,\square,\top\}$ defined by \eqref{equa:def S_K-1} and \eqref{equa:def S_K-2}. We already know that $\S_{\square\DN}$ satisfies properties (A1)-(A4). Moreover, it is straightforward to show directly that the following conditions hold:
\begin{enumerate}
	\item[(N)] $\vdash_{\S_{\square\DN}}\ph$ implies $\vdash_{\S_{\square\DN}}\square\ph$;
	\item[$(\square m)$] $\square m(\ph,\psi,\chi)\vdash_{\S_{\square\DN}}m(\square(\ph\vee\chi),\square(\psi\vee\chi),\square\chi)$;
	\item[$(m\square)$] $m(\square(\ph\vee\chi),\square(\psi\vee\chi),\square\chi)\vdash_{\S_{\square\DN}}\square m(\ph,\psi,\chi)$.
\end{enumerate}
Moreover, notice that $\K_{\S_{\square\DN}}=\Alg(\S_{\square\DN})=\square\DN$ (see on page \pageref{AlgS_K=K}). Now let us show that $\S_{\square\DN}$ is the weakest selfextensional non-pseudo axiomatic logic satisfying conditions (A1)-(A4), (N), $(\square m)$ and $(m\square)$.

\begin{proposition}
Let $\S$ be a selfextensional non-pseudo axiomatic logic satisfying conditions \textup{(B1)-(A4), (N)}, $(\square m)$ and $(m\square)$. Then, $\S$ is an extension of $\S_{\square\DN}$.
\end{proposition}

\begin{proof}
Let $\S$ be a selfextensional non-pseudo axiomatic logic satisfying conditions \textup{(A1)-(A4), (N)}, $(\square m)$ and $(m\square)$. Then, by Theorems \ref{theo:dn-term iff dn-based} and \ref{theo:AlgS variety}, we have that $\S$ is $\dn$-based relative to $\Alg(\S)=\K_\S$. Since $\S$ satisfies conditions (N), $(\square m)$ and $(m\square)$, it is straightforward to show that $\K_\S\subseteq\square\DN$. Hence, since $\K_\S\subseteq\square\DN=\K_{\S_{\square\DN}}$, we obtain by Theorem \ref{theo:S isom K} that $\S_{\square\DN}\leq \S$.
\end{proof}

\section{The truth-preserving logic and the logic preserving degrees of truth}\label{sec:5}

Let $\S$ be a selfextensional logic with a $\dn$-term $m$. We know that the canonical class  of algebras $\Alg(\S)$ associated with $\S$ is a $\dn$-based variety, that is, $\Alg(\S)$ is a variety and for every $A\in\Alg(\S)$, the $\{m\}$-reduct $\langle A,m^A\rangle$ is a distributive nearlattice. Hence, every algebra $A$ in $\Alg(\S)$ has associated a partial order. Thus, a sentential logic can be defined by this class of algebras: a formula $\ph$ is a logical consequence from some premises if and only if for every algebra and every interpretation, whenever the interpretation of the premises have a common lower bound, the interpretation of the formula $\ph$ also has the same lower bound; the so-called \textit{the logic that preserves degrees of truth}.

Let now $\S$ be a selfextensional logic with a $\dn$-term $m$ and with theorems. Thus, for every $A\in\Alg(\S)$, the distributive nearlattice $\langle A,m^A\rangle$ has a greatest element $h\ph$, where $\ph$ is any theorem of $\S$ and $h$ is any interpretation on $A$. We denote this element, for every algebra $A$, by $1^A$. If for every algebra $A$ we consider that $1^A$ is the only truth value (\textit{the truth}), then it can be defined a logic as follows:  a formula $\ph$ follows logically from some premises if and only if for every algebra and every interpretation, whenever the interpretation of the premises are true, then the interpretation of the formula $\ph$ is true. This logic is known in the literature as \textit{the truth-preserving logic} (or as \textit{the assertional logic}) associated with $\Alg(\S)$.

In this section, given a selfextensional logic $\S$ with a $\dn$-term (and with theorems), we study the connections between this logic $\S$ and both the logic that preserves degrees of truth and the  truth-preserving logic associated with the class of algebras $\Alg(\S)$.

\subsection{The logic preserving degrees of truth}

\textit{Logics  preserving degrees of truth} associated with some particular ordered algebraic structures are studied and discussed in several articles, for instance see \cite{Fo03,FoGiToVe06,No90,Fo09}. In particular, \cite{Fo03} is an interesting contribution to the discussion on the role of degrees of truth in many-valued logics from the point of view of Abstract Algebraic Logic.

Let $\K$ be a class of algebras such that every algebra $A\in\K$ has a partial order $\leq$ associated with its universe. Thus, \textit{the logic preserving degrees of truth} with respect to $\K$ is defined as follows:
let $\ph_1,\dots,\ph_n,\ph\in\Fm$,
\begin{equation}\label{equa:logic preserving degrees 1}
\begin{split}
\ph_1,\dots,\ph_n\vdash_\K^\leq\ph \iff 	&(\forall A\in \K)(\forall h\in\Hom(\Fm,A))\\
										&(\forall a\in A)(\text{if } a\leq h\ph_i \ \forall i=1,\dots,n, \text{ then }  a\leq h\ph)
\end{split}
\end{equation}
and
\begin{equation}\label{equa:logic preserving degrees 2}
\begin{split}
\emptyset\vdash_\K^\leq\ph
\iff
(\forall A\in\K)(\forall h\in\Hom(\Fm,A))
 (\forall a\in A)(a\leq h\ph).
\end{split}
\end{equation}
For an arbitrary $\Gamma\subseteq\Fm$, 
\begin{equation}\label{equa:logic preserving degrees 3}
\Gamma\vdash_\K^\leq\ph \iff \text{ there is a finite } \Gamma_0\subseteq\Gamma \text{ such that } \Gamma_0\vdash_\K^\leq\ph.
\end{equation}

\begin{definition}\label{def:S_K^<}
Let $\K$ be a $\dn$-based class of algebras. The \textit{logic preserving degrees of truth with respect to} $\K$, $\S_\K^\leq=\langle \Fm,\vdash_\K^\leq\rangle$, is defined by \eqref{equa:logic preserving degrees 1}--\eqref{equa:logic preserving degrees 3}. 
\end{definition}

\begin{proposition}\label{prop:Sl extension of SK}
The logic $\S_\K^\leq$ is an extension of the logic $\S_\K$. Moreover, $\S_\K$ and $\S_\K^\leq$ have the same theorems, that is, $\vdash_\K\ph$ if and only if $\vdash_\K^\leq\ph$.
\end{proposition}

Let $\K$ be a $\dn$-based class. From \eqref{equa:logic preserving degrees 1}, it is clear that for all $\ph,\psi\in\Fm$, 
\[
\ph\vdash_\K^{\leq}\psi \iff (\forall A\in\K)(\forall h\in\Hom(\Fm,A))(h\ph\leq h\psi).
\]
Thus, the following proposition is straightforward. Recall that for a logic $\S$, $\K_\S$ denotes the intrinsic variety of $\S$.

\begin{proposition}\label{prop:properties of Sl}
Let $\K$ be a $\dn$-based class. Then,
\begin{enumerate}[{\normalfont (1)}]	
	\item the logic $\Sl$ is selfextensional;
	\item $\mathbb{V}(\K)=\K_{\Sl}$;
	\item $\Sl$ satisfies properties \normalfont{(A1)-(A3)}.
\end{enumerate}
\end{proposition}

\begin{example}\label{exa:Sl not A4}
The logic $\Sl$ does not satisfy property (A4), and thus the logics $\Sl$ and $\S_\K$ are different. Consider the language $\Ll:=\{m,\bot_1,\bot_2,\top\}$ of type $(3,0,0,0)$. Let $A$ be the distributive nearlattice given in Figure \ref{fig:not hold A4}, and such that $\bot_1^A=a$, $\bot_2^A=b$ and $\top^A=1$. Let $x$ be an arbitrary variable. Then, we have $\bot_1,\bot_2\vdash_A^{\leq}x$ but $m(\bot_1,\bot_2,x)\nvdash_A^{\leq}x$. Hence, $\Sl$ does not satisfy property (A4). 
\end{example}

\begin{figure}
\begin{tikzpicture}[scale=.5,inner sep=.5mm]
\node[bull] (left) at (-1.75,2.5) [label=below:$a$] {};%
\node[bull] (midle) at (0,2.5) [label=below:$b$] {};%
\node[bull] (right) at (1.75,2.5) [label=below:$c$] {};%
\node[bull] (top) at (0,5) [label=above:$1$] {};%

\draw (left) -- (top) -- (right) -- (top) -- (midle) ;%
\end{tikzpicture}
\caption{The distributive nearlattice of Example \ref{exa:Sl not A4}.}
\label{fig:not hold A4}
\end{figure}
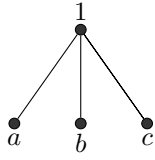

\begin{proposition}\label{prop:(A4) implies S_K=Sl}
Let $\K$ be a $\dn$-based class. If the logic $\Sl$ satisfies property \normalfont{(A4)}, then $\S_\K=\Sl$.
\end{proposition}

\begin{proof}
We need only prove that $\Sl\leq\S_\K$. Let $\ph_0,\dots,\ph_n,\ph\in\Fm$. Assume that $\ph_0,\dots,\ph_n\vdash_\K^{\leq}\ph$. Thus, by (A4), $m^{n}(\ph_0,\dots,\ph_n,\ph)\vdash_\K^{\leq}\ph$. Then, it is straightforward to check that $m^{n}(\ph_0,\dots,\ph_n,\ph)\vdash_\K\ph$. Now, by condition (3) of Proposition \ref{prop:consequence from (A1)-(A4)}, it follows that $\ph_0,\dots,\ph_n\vdash_\K\ph$. Therefore, by Proposition \ref{prop:Sl extension of SK}, we obtain that $\Sl=\S_\K$.
\end{proof}

\begin{proposition}\label{prop:K=Alg(S_K)}
If $\K$ is a $\dn$-based variety, then 
\[
\K_{\Sl}=\K=\K_{\S_\K}=\Alg(\S_\K).
\]
\end{proposition}

\begin{proof}
It follows from Propositions \ref{prop:properties of Sl} and \ref{prop:K_S_K=K}, and Theorem \ref{theo:AlgS variety}.
\end{proof}

For the next results, we need the following concepts. Let $P$ be a partially ordered set. Let $X\subseteq P$. We set $X^{\ell}:=\{a\in P: a\leq x, \text{ for all } x\in X\}$ and $X^u:=\{a\in P: a\geq x, \text{ for all } x\in X\}$. A subset $F\subseteq P$ is said to be a \textit{Frink filter} (\cite{Fri54}) if for every finite $X\subseteq F$, we have $X^{\ell u}\subseteq F$. We denote by $\FiF(P)$ the collection of all Frink filters of $P$. It is easy to check that $\FiF(P)$ is an algebraic closure system on $P$.

Then, condition \eqref{equa:logic preserving degrees 1} can be written as
\begin{equation*}
\begin{split}
\ph_1,\dots,\ph_n\vdash_\K^\leq\ph \iff 	&(\forall A\in \K)(\forall h\in\Hom(\Fm,A))\\
										&h\ph\in\{h\ph_1,\dots,h\ph_n\}^{\ell u}.
\end{split}
\end{equation*}

The next lemma is straightforward, and thus we omit its proof.

\begin{lemma}\label{lem:FiF < Fi_S}
For every $\dn$-class $\K$ and every algebra $A\in\K$, $\FiF(A)\subseteq\Fi_{\Sl}(A)$.
\end{lemma}

\begin{proposition}\label{prop:Alg(S_K)=Alg(S_K<)}
Let $\K$ be a $\dn$-based variety. Then, we have 
\[
\Alg(\Sl)=\K=\Alg(\S_\K).
\]
\end{proposition}

\begin{proof}
By Lemma \ref{lem:AlgS<K_S}, we know that $\Alg(\Sl)\subseteq \K_{\Sl}=\K$. Recall that $A\in\Alg(\Sl)$ if and only if there is an algebraic closure system $\Cs$ on $A$ such that $\langle A,\Cs\rangle\in\GMod^*(\Sl)$. Let $A\in\K$. Since $\FiF(A)$ is an algebraic closure system and $\FiF(A)\subseteq\Fi_{\Sl}(A)$, it follows that $\langle A,\FiF(A)\rangle\in\GMod(\Sl)$. 

Notice that the upsets $[a)=\{x\in A: a\leq x\}$ are Frink filters of $A$. So, we have $\langle a,b\rangle\in\Lambda_A(\FiF(A))$ if and only if $a=b$. Then, the Frege relation $\Lambda_A(\FiF(A))$ of the g-model $\langle A,\FiF(A)\rangle$ is the identity relation. Hence, $\langle A,\FiF(A)\rangle\in\GMod^*(\Sl)$; and thus, $A\in\Alg(\Sl)$. Therefore, by Proposition \ref{prop:K=Alg(S_K)}, $\Alg(\Sl)=\K=\Alg(\S_\K)$.
\end{proof}

The following corollary is a consequence of Lemma \ref{lem:FiF < Fi_S}.

\begin{corollary}
Let $\K$ be a $\dn$-based variety. Then, the logic $\Sl$ is fully selfextensional.
\end{corollary}

Let $\K$ be a $\dn$-based class and $A\in\K$. Since $\FiF(A)$ is a closure system, $\FiF(A)$ is a lattice. Let us show that the distributivity of the lattices $\FiF(A)$ is a sufficient condition for  logics $\S_\K$ and $\Sl$ coincide.

\begin{proposition}\label{prop:FIF distr implies S_K=Sl}
Let $\K$ be a $\dn$-based variety. If for every $A\in\K$ the lattice $\FiF(A)$ is distributive, then $\S_\K=\Sl$.
\end{proposition}

\begin{proof}
Assume that for each $A\in\K$, the lattice $\FiF(A)$ is distributive. Since the logic $\Sl$ is an extension of the logic $\S_\K$, we have $\Fi_{\Sl}(A)\subseteq\Fi_{\S_\K}(A)$ for every algebra $A$. Let now $A\in\K$. Since the lattice $\FiF(A)$ is distributive, it follows that $\FiF(A)=\Fi(A)$ (see \cite[Proposition 4.3]{Go17a}). Then, by Proposition \ref{prop:Fi_S(A)=Fi(A)} and Lemma \ref{lem:FiF < Fi_S}, it follows that 
\[
\Fi_{\S_\K}(A)\setminus\{\emptyset\}=\Fi(A)=\FiF(A)\subseteq\Fi_{\Sl}(A).
\]
Hence, since $\S_\K$ and $\Sl$ have the same theorems, we obtain $\Fi_{\S_\K}(A)=\Fi_{\Sl}(A)$ for all $A\in\K$. Now, notice that every sentential logic $\S$ is complete with respect to the class of g-models $\{\langle A,\Fi_\S(A)\rangle: A\in\Alg(\S)\}$. Thus, by Proposition \ref{prop:Alg(S_K)=Alg(S_K<)}, we obtain that $\S_\K=\Sl$.
\end{proof}

\subsection{The truth-preserving logic}\label{subsec:truth-preserving logic}

We will say that a class of algebras $\K$ is $\dn^1$-\textit{based} if it is $\dn$-based and every algebra $A$ in $\K$ has a greatest element, that is, there is $1^A\in A$ such that $a\leq 1^A$ for all $a\in A$.  Notice that if $\S$ is a selfextensional logic with a $\dn$-term and with theorems, then $\Alg(\S)$ is a $\dn^1$-based variety and for every $A\in\Alg(\S)$, $1^A=h\ph$ for any theorem $\ph$ and any interpretation $h$ on $A$. Moreover, since $\S$ has theorem, it follows by Proposition \ref{prop:K_S_K=K} that $\S=\S_{\Alg(\S)}$. From now on, unless otherwise stated, $\K$ will denote a $\dn^1$-based variety.

\begin{definition}
The \textit{truth-preserving logic associated with} $\K$, $\S_\K^1=\langle\Fm,\vdash_\K^1\rangle$, is defined as follows: let $\Gamma\cup\{\ph\}\subseteq\Fm$ be finite,
\begin{equation}
\begin{split}
\Gamma\vdash_\K^1\ph\iff &(\forall A\in \K)(\forall h\in\Hom(\Fm,A))\\
									&(h\psi=1^A \text{ for all } \psi\in\Gamma \implies h\ph=1^A).
\end{split}
\end{equation}
For an arbitrary $\Gamma\subseteq\Fm$, $\Gamma\vdash_\K^1\ph$ if and only if there is a finite $\Gamma_0\subseteq\Gamma$ such that $\Gamma_0\vdash_\K^1\ph$.
\end{definition}

\begin{proposition}\label{prop:S1 is an extension}
The logic $\S_\K^1$ is an extension of $\Sl$, and hence it is an extension of the logic $\S_\K$. Moreover, logics $\S_\K$, $\Sl$ and $\S_\K^1$ have the same theorems.
\end{proposition}

Let $A$ be an algebra and $F\subseteq A$. The \textit{Leibniz congruence of} $F$ \textit{relative to} $A$, denoted by $\Omega_AF$, is the greatest congruence of $A$ compatible with $F$, that is, it does not relate elements in $F$ with elements not in $F$. The mapping $F\mapsto\Omega_AF$ is called the \textit{Leibniz operator} of the algebra $A$ and it is denoted by $\Omega_A$. This operator is an essential tool in Abstract Algebraic Logic for classifying sentential logics.  The structure of this classification of sentential logics is called the \textit{Leibniz hierarchy}. We address the reader to \cite{Cz01,FoJaPi03,Fo16} for more information on this hierarchy.

\begin{proposition}\label{prop:S_K=S_K^1}
If for every $A\in\K$ and every nonempty and finite $X\subseteq A$ we have $\Fig_A(X)=\{a\in A: \langle a,1^A\rangle\in\Omega_A\Fig_A(X)\}$, then $\S_\K=\S_\K^1$.
\end{proposition}

\begin{proof}
By Proposition \ref{prop:S1 is an extension}, we only need to prove that for all $\ph_0,\dots,\ph_n,\ph\in\Fm$, $\ph_0,\dots,\ph_n\vdash_\K^1\ph$ implies $\ph_0,\dots,\ph_n\vdash_\K\ph$. So, assume that $\ph_0,\dots,\ph_n\vdash_\K^1\ph$. Let $A\in\K$ and $h\in\Hom(\Fm,A)$. We set $F:=\Fig_A(h\ph_0,\dots,h\ph_n)$. Since $\K$ is a variety and $\Omega_A F$ is a congruence on $A$, it follows that $A/\Omega_AF\in\K$. By hypothesis, we have that $F=1^A/\Omega_AF=1^{A/\Omega_AF}$.  Let now $\widehat{h}:=\pi\circ h\colon\Fm\to A/\Omega_AF$, where $\pi\colon A\to A/\Omega_AF$ is the natural map. Thus, $\widehat{h}\ph_i=h\ph_i/\Omega_AF=1^A/\Omega_AF=1^{A/\Omega_AF}$ for all $i=0,1,\dots,n$. Then, $\widehat{h}\ph=1^{A/\Omega_AF}$.  Hence $h\ph\in F=\Fig_A(h\ph_0,\dots,h\ph_n)$. Thus, by Proposition \ref{prop:properties of m^n}, we obtain that $h\ph=m^{n}(h\ph_0,\dots,h\ph_n,h\ph)$. Hence, $\ph_0,\dots,\ph_n\vdash_\K\ph$.
\end{proof}


One of the most important and large classes of sentential logics under the point of view of Abstract Algebraic Logic is the class of \textit{protoalgebraic logics}. This class of logics was introduced and studied by Blok and Pigozzi \cite{BloPi86}, and independently it was considered by Czelakowski \cite{Cz85}. There are several useful characterizations of the notion of protoalgebraibility. For our purposes, we choose the following as the definition of protoalgebraic logic.

\begin{definition}
A sentential logic $\S$ is said to be \textit{protoalgebraic} if there is a set of formulas in two variables $\Delta(x,y)$ such that $\emptyset\vdash_\S\Delta(x,x)$ and $x,\Delta(x,y)\vdash_\S y$. A set with these two properties will be called a \textit{set of protoimplication formulas} for $\S$.
\end{definition}
 
Algebraizable logics, finitely algebraizable logics and regularly algebraizable logics are important classes of protoalgebraic logics, see for instance \cite{BloPi89,Cz01,FoJaPi03,Fo16}. A sentential logic $\S$ is \textit{algebraizable} if and only if (i) there is a set of formulas in two variables $\Delta(x,y)$ such that for each algebra $A$ and each $\S$-filter $F$ of $A$, $ \Omega_AF=\{\langle a,b\rangle\in A^2: \Delta^A(a,b)\subseteq F\}$, and (ii) there is a set of equations $\tau(x)$ in one variable such that for every algebra $A$ and every $\S$-filter $F$ of $A$ with $\Omega_AF$ being the identity relation, $F=\{a\in A: A\models\tau(x)[a]\}$. The set $\Delta(x,y)$ is called a set of \textit{equivalence formulas} for $\S$. It follows that every set of equivalence formulas is a set of protoimplication formulas.

A logic $\S$ is said to be \textit{finitely algebraizable} when it is algebraizable and the sets $\Delta(x,y)$ and $\tau(x)$ are finite. A logic $\S$ is said to be \textit{regularly algebraizable} if it is finitely algebraizable and for every set of equivalence formulas $\Delta(x,y)$ the G-rule $x,y\vdash_\S\Delta(x,y)$ is satisfied.

A quasivariety $\K$ is said to be \textit{pointed} if there is a term $\ph(x_1,\dots,x_n)$ with the property that $\ph(x_1,\dots,x_n)\approx\ph(y_1,\dots,y_n)$ is valid in $\K$ for all variables $y_1,\dots,y_n$. Such a term is called a \textit{constant term} since it behaves like a constant. Once we fix a constant term, we will denote it by 1; we will say that $\K$ is \textit{1-pointed} and we will use $1^A$ to refer the interpretation of the constant term 1 in $A$, for each $A\in\K$. Notice that if $\S$ is a selfextensional logic with a $\dn$-term and with theorems, then the variety $\Alg(\S)$ is 1-pointed, where the constant term is any theorem of $\S$, and for each $A\in\Alg(\S)$, $1^A$ is the greatest element of $\langle A,m^A\rangle$.

A 1-pointed quasivariety $\K$ is said to be \textit{relatively point regular} when for every $A\in\K$ and every $\theta,\theta'\in\Con_\K(A)$, if $1^A/\theta=1^A/\theta'$, then $\theta=\theta'$. If $\K$ is a variety, then  we  say simply that $\K$ is \textit{point regular}, since $\Con_\K(A)=\Con(A)$ for every $A\in\K$.

\begin{theorem}[\cite{CzPi04}]
A logic $\S$ is regularly algebraizable if and only if $\Alg(\S)$ is 1-pointed and relatively point regular quasivariety and $\S=\S_{\Alg(\S)}^1$.
\end{theorem}

Let $\langle A,m\rangle$ be a distributive nearlattice. Then, $\theta$ is a congruence on $A$ if and only if (i) if $\langle a,b\rangle,\langle c,d\rangle\in\theta$, then $\langle a\vee c,b\vee d\rangle\in\theta$, and (ii) if $\langle a,b\rangle,\langle c,d\rangle\in\theta$ and $a\wedge c, b\wedge d$ exist in $A$, then $\langle a\wedge c,b\wedge d\rangle\in\theta$, see \cite{Hi80}.

\begin{theorem}\label{theo:S=S_Alg^1}
Let $\S$ be a selfextensional logic with a $\dn$-term and with theorems. If the logic $\S$ is protoalgebraic, then $\S=\S_{\Alg(\S)}^1$ if and only if $\S$ is regularly algebraizable.
\end{theorem}

\begin{proof}
Let $\S$ be a selfextensional logic with a $\dn$-term and with theorems, and assume that $\S$ is protoalgebraic. Recall that $\S=\S_{\Alg(\S)}$, and by Proposition \ref{prop:S1 is an extension} we have that $\S\subseteq\S_{\Alg(\S)}^1$. If $\S$ is regularly algebraizable, then by the previous theorem we have that $\S=\S_{\Alg(\S)}^1$. 

Now, conversely, assume that $\S=\S_{\Alg(\S)}^1$. We already know that $\Alg(\S)$ is 1-pointed. To show that the logic $\S$ is regularly algebraizable we only need to prove, by the previous theorem, that the variety $\Alg(\S)$ is point regular. Let $A\in\Alg(\S)$ and let $\theta,\theta'\in\Con(A)$ be such that $1^A/\theta=1^A/\theta'$. As $\S$ is finitary and protoalgebraic, let $\Delta(x,y)$ be a finite set of protoimplication formulas for $\S$. We set $\Delta(x,y)=\{\ph_1(x,y),\dots,\ph_n(x,y)\}$. Let $\langle a,b\rangle\in\theta$. So, 
\begin{equation}\label{equa:aux 4}
\{\langle\ph^A(a,a),\ph^A(a,b)\rangle: \ph(x,y)\in\Delta(x,y)\}\subseteq\theta.
\end{equation}
Since $\vdash_\S\Delta(x,x)$, it follows that $\{\langle 1^A,\ph^A(a,b)\rangle: \ph(x,y)\in\Delta(x,y)\}\subseteq\theta$. Thus,  $\{\langle 1^A,\ph^A(a,b)\vee b\rangle: \ph(x,y)\in\Delta(x,y)\}\subseteq\theta$. Then, by hypothesis, $\{\langle 1^A,\ph^A(a,b)\vee b\rangle: \ph(x,y)\in\Delta(x,y)\}\subseteq\theta'$. It follows that
\[
\langle 1^A,(\ph^A_1(a,b)\vee b)\wedge\dots\wedge(\ph^A_n(a,b)\vee b)\rangle\in\theta'.
\]
Then, we obtain 
\[
\langle a\vee b,(\ph^A_1(a,b)\vee b)\wedge\dots\wedge(\ph^A_n(a,b)\vee b)\wedge(a\vee b)\rangle\in\theta'.
\]
Notice that 
\[
(\ph^A_1(a,b)\vee b)\wedge\dots\wedge(\ph^A_n(a,b)\vee b)\wedge(a\vee b)=m^n(\ph^A_1(a,b),\dots,\ph^A_n(a,b),a,b).
\]
Thus,
\begin{equation}\label{equa:aux 5}
\langle a\vee b, m^n(\ph^A_1(a,b),\dots,\ph^A_n(a,b),a,b)\rangle\in\theta'.
\end{equation}
Now, since $x,\Delta(x,y)\vdash_\S y$, it follows by property (A4) that 
\[
m^n(\ph_1(x,y),\dots,\ph_n(x,y),x,y)\vdash_\S y.
\]
Let $h\in\Hom(\Fm,A)$ be such that $h(x)=a$ and $h(y)=b$. Then, we obtain that
\[
m^n(\ph^A_1(a,b),\dots,\ph^A_n(a,b),a,b)=b.
\]
Thus, by \eqref{equa:aux 5}, we have $\langle a\vee b,b\rangle\in\theta'$. With a similar argumentation we can get $\langle a\vee b,a\rangle\in\theta'$. Hence, $\langle a,b\rangle\in\theta'$. We have proved that $\theta\subseteq\theta'$. Similarly, we have $\theta'\subseteq\theta$. Then $\theta=\theta'$. Hence, $\Alg(\S)$ is point regular. Therefore, since $\Alg(\S)$ is 1-pointed and point regular and since $\S=\S_{\Alg(\S)}^1$, it follows by the previous theorem that $\S$ is regularly algebraizable.
\end{proof}

\section*{Conclusions}

Given an algebraic language $\Ll$ with a ternary term $m$, we have defined when $m$ is a distributive nearlattice term ($\dn$-term) for a sentential logic $\S$ (see Definition \ref{def:dn-term}). The term $m$ is a $\dn$-term for $\S$ if it satisfies some syntactical properties ((A1)--(A4)); roughly speaking, these properties (A1)-(A4) mean that when $m$ is interpreted in every algebra $A$ of the algebraic counterpart of $\S$ the $\{m\}$-reduct $\langle A,m^A\rangle$ is a distributive nearlattice.

Then, we characterised the selfextensional logics with a $\dn$-term $m$ as those for which the consequence relation can be defined by the order induced by the ternary term $m$ interpreted in the algebras of the algebraic counterpart of the logic (see Theorems \ref{theo:dn-term iff dn-based} and \ref{theo:AlgS variety}).

Given a $\dn$-based variety $\K$, we define the logic $\S_\K$ (see on page \pageref{def:logic S_K}); this logic is selfextensional with a $\dn$-term, and $\K=\Alg(\S_\K)$. Since the algebras of $\K$ have associated a partial order, it can be defined the logic preserving degrees of truth $\Sl$ (see Definition \ref{def:S_K^<}). We have shown some properties of the logic $\Sl$. We found some sufficient conditions for logics $\S_\K$ and $\Sl$ coincide (see Propositions \ref{prop:(A4) implies S_K=Sl} and \ref{prop:FIF distr implies S_K=Sl}).
 
If $\S$ is a selfextensional logic with a $\dn$-term and with theorems, then the algebraic counterpart $\Alg(\S)$ of $\S$ is 1-pointed. Thus, we can define the truth-preserving logic $\S_{\Alg(\S)}^1$ associated with $\Alg(\S)$. We also have found some sufficient conditions for logics $\S$ and $\S_{\Alg(\S)}^1$ coincide (see Proposition \ref{prop:S_K=S_K^1} and Theorem \ref{theo:S=S_Alg^1}).

Selfextensional logics having a DN-term (or equivalently, DN-based logics) include those sentential logics defined by varieties of distributive lattices. More precisely, consider an algebraic language $\Ll$ with two binary terms $\wedge$ and $\vee$ and let $\K$ be a variety of type $\Ll$ such that for each $A\in\K$, $\langle A,\wedge^A,\vee^A\rangle$ is a distributive lattice. Then, it is clear that $\K$ is, in particular, a DN-based variety. On the one hand, since the algebras of $\K$ have a conjunction, it can be defined in a natural way a logic $\S_\K'$ as in \cite[pp. 79]{Ja06} which is the semilattice-based logic associated with $\K$. On the other hand, since $\K$ is a DN-based variety, we can consider the DN-based logic $\S_\K$ associated with $\K$ (see \eqref{equa:def S_K-1} and \eqref{equa:def S_K-2}). Then, it is straightforward to check that the two logics $\S_\K'$ and $\S_\K$ coincide.

Let $\K$ be a $\dn^1$-based variety (see Subsection \ref{subsec:truth-preserving logic}). Since distributive nearlattices can be considered as  generalisations both of Tarski algebras and of distributive lattices, it will be important to carry out studies on the logics $\S_\K$, $\Sl$ and $\S_\K^1$ under the point of view of AAL. In particular, we will study how these logics behave concerning the classifications in the hierarchy of Leibniz and the hierarchy of Frege. These task will be pursued elsewhere. The papers \cite{Ja12,FoJa01,BoEsFoGiGoToVe09,Fo11} will be of great help to carry out these investigations.


\end{document}